\documentclass[oneside,english]{amsart}
\usepackage[T1]{fontenc}
\usepackage[latin9]{inputenc}
\usepackage{babel}
\usepackage{amsthm}
\usepackage{amstext}
\usepackage{amssymb}
\usepackage{esint}
\usepackage[unicode=true,pdfusetitle,
 bookmarks=true,bookmarksnumbered=false,bookmarksopen=false,
 breaklinks=false,pdfborder={0 0 1},backref=false,colorlinks=false]
 {hyperref}
\usepackage{breakurl}

\makeatletter
\numberwithin{equation}{section}
\numberwithin{figure}{section}
\theoremstyle{plain}
\newtheorem{thm}{\protect\theoremname}
  \theoremstyle{plain}
  \newtheorem{prop}[thm]{\protect\propositionname}
  \theoremstyle{definition}
  \newtheorem{defn}[thm]{\protect\definitionname}
  \theoremstyle{remark}
  \newtheorem{rem}[thm]{\protect\remarkname}
  \theoremstyle{plain}
  \newtheorem{lem}[thm]{\protect\lemmaname}
  \theoremstyle{plain}
  \newtheorem{cor}[thm]{\protect\corollaryname}
  \theoremstyle{definition}
  \newtheorem{example}[thm]{\protect\examplename}

\usepackage[a4paper]{geometry}
\usepackage{breqn}

\makeatother

  \providecommand{\corollaryname}{Corollary}
  \providecommand{\definitionname}{Definition}
  \providecommand{\examplename}{Example}
  \providecommand{\lemmaname}{Lemma}
  \providecommand{\propositionname}{Proposition}
  \providecommand{\remarkname}{Remark}
\providecommand{\theoremname}{Theorem}

\begin{document}

\title{A Chen-Fliess Approximation for diffusion functionals}

\author{Christian Litterer}

\author{Harald Oberhauser}

\address{TU Berlin, Fakultät II, Institut für Mathematik, MA 7-2, Strasse
des 17. Juni 136, 10623 Berlin.}

\email{h.oberhauser@gmail.com}

\begin{abstract}
We show that an interesting class of functionals of stochastic differential
equations can be approximated by a Chen-Fliess series of iterated
stochastic integrals and give a $L^{2}$ error estimate, thus generalizing
the standard stochastic Taylor expansion. The coefficients in this
series are given a very intuitive meaning by using functional derivatives,
recently introduced by B. Dupire.
\end{abstract}
\maketitle

\section{Introduction}

The expansion of the SDE solution 
\begin{equation}
dY_{t}=\sum_{i=1}^{d}V_{i}\left(Y_{t}\right)\circ dB_{t}^{i},\, Y\left(0\right)\in\mathbb{R}^{e}\label{eq:SDE-1}
\end{equation}
in a stochastic (Stratonovich) Taylor series 
\begin{equation}
f\left(Y_{t}\right)-f\left(Y_{s}\right)=\sum_{k=1}^{N}\sum_{\left(i_{1},\ldots,i_{k}\right)\in\left\{ 1,\ldots,d\right\} ^{k}}V_{i_{1}}\cdots V_{i_{k}}f\left(Y_{s}\right)\int_{\Delta^{k}\left(s,t\right)}\circ dB^{\left(i_{1},\ldots,i_{k}\right)}+R_{N}\left(s,t\right)\label{eq:stoch taylor}
\end{equation}
is an important tool in the numerical analysis of stochastic differential
equations, cf. \cite{kloeden-platen-92}. In this short note we consider
a Stratonovich stochastic Taylor expansion in the spirit of Kloeden
and Platen \cite{kloeden-platen-92} and generalize the classic formula
(\ref{eq:stoch taylor}) that applies to smooth functions  $f:\mathbb{R}^{e}\rightarrow\mathbb{R}$
by giving meaning to such an approximation when $f$ is replaced by
a sufficiently regular, nonanticipative functional 
\[
F:\left[0,T\right]\times C\left(\left[0,T\right],\mathbb{R}^{e}\right)\rightarrow\mathbb{R}
\]
and derive a $L^{2}$ bound for the remainder of the expansion.  The
crucial tool in obtaining these results are ideas of B. Dupire \cite{Duprire_FIto}
on derivatives of nonanticipative functionals (cf.~the extensions
of Cont\&Fournie \cite{ContFournie,ContFournie_Pathwise} and the
work on pathdependent viscosity PDEs \cite{2011arXiv1106.1144P,2011arXiv1108.4317P,2011arXiv1109.5971E}).
Dupire showed that one can define a time derivative $\partial_{t}F$
and a space derivative $\nabla F$ of $F$ such that one arrives at
the approximation (for simplicity assume $e=1$) 
\[
dF\approx\partial_{t}Fdt+\nabla FdB_{t}+\frac{\nabla^{2}F}{2}\left(dB_{t}\right)^{2}.
\]
The intuition behind these derivatives is that the time derivative
measures the infinitesimal drift of $F$ whereas $\nabla F$ measures
the response of $F$ to instantaneous changes in the underlying path.
In section \ref{sec:Approximating-Functionals-of} we show that these
derivatives allow to give proper meaning to the expression $V_{i_{1}}\cdots V_{i_{k}}F\left(s,Y\right)$
if one interprets a vector field in a corresponding way as a derivation
in terms of these functional derivatives (e.g. if $F\left(t,x\right)=f\left(x_{t}\right)$
this simply coincides with the standard notion of $V_{i_{1}}\cdots V_{i_{k}}$
as a differential operator). This subsequently allows us to adapt
the standard estimates for (\ref{eq:stoch taylor}) in the functional
setting.

The functional perspective of the stochastic Taylor expansion suggests
some interesting connections to problems studied in deterministic,
classic control theory: many systems are modeled as a nonlinear, causal
response to a $d$-dimensional input $\left(u^{1},\ldots,u^{d}\right)\in C\left(\left[0,T\right],\mathbb{R}^{d}\right)$
and control theory provides a powerful toolbox to deal with functionals
of the form 
\[
F:\left[0,T\right]\times C\left(\left[0,T\right],\mathbb{R}^{d}\right)\rightarrow\mathbb{R}
\]
which are nonanticipative in the sense that $F\left(t,u^{1},\ldots,u^{d}\right)$
depends only on values of $\left(u^{1},\ldots,u^{d}\right)$ up to
time $t$ and depend continuously on the input $\left(u^{1},\ldots,u^{d}\right)$.
The basic example being the solution map of a controlled differential
equation 
\begin{equation}
\frac{dy}{dt}\left(t\right)=\sum_{i=1}^{d}V_{i}\left(y\left(t\right)\right)u^{i}\left(t\right),y\left(0\right)\in\mathbb{R}\label{eq:control}
\end{equation}
i.e. $F\left(t,u^{1},\ldots,u^{d}\right)=y\left(t\right)$. We mention
- pars pro toto - the work of Brockett, Fliess and Sussmann, \cite{RogerW1976167,fliess1981fonctionnelles,sussmann1975semigroup},
which emphasized the importance of the signature of a path in the
sense of Chen \cite{chen-57,chen-58} as a central object in control
theory.  It is a well known fact to any system theorist or electrical
engineer that a functional $F$ that is linear 
\[
F\left(t,\sum_{i=1}^{d}\alpha_{i}u^{i}+\beta_{i}v^{i}\right)=\sum_{i=1}^{d}\alpha_{i}F\left(t,u^{i}\right)+\beta_{i}F\left(t,v^{i}\right)
\]
and additionally time-invariant is completely described by its response
to a Dirac impulse $\delta$ at time $0$ (clearly, the functional
given by (\ref{eq:control}) does not fit into this framework). As
mentioned above, Dupire's functional calculus rests on two derivatives
and the space derivative is, at least formally, closely related to
the well known method in control theory of response to a Dirac delta
(however in integrated form ``$\int_{0}^{t}\left(u_{r}+\epsilon\delta_{t}\right)dr=\int_{0}^{t}u_{r}dr+\epsilon,u=\frac{dB}{dt}$''
and at running time). Our interest lies in the stochastic case, but
these functional derivatives might already be of interest to the classic
control theory case for non-linear, non-time-invariant systems of
the form (\ref{eq:control}). In this context we recall in section
\ref{sec:control} a result of M. Fliess, \cite{fliess1981fonctionnelles}
on the existence of approximations to continuous, nonanticipative
functionals (this result of Fliess is the reason why we refer to the
expansion in section \ref{sec:Approximating-Functionals-of} as a
Chen-Fliess approximation instead of stochastic Taylor expansion for
functionals).

Of course, one can see the finite expansion (\ref{eq:stoch taylor})
as a small piece in the seminal work initiated by Azencott and Ben
Arous \cite{azencott-82,benarous-89} on asymptotic expansions of
SDE flows and subsequent generalizations of e.g. Castell and Hu \cite{castell,hu}
using the generalized Campbell-Baker-Hausdorff formula of Strichartz
\cite{strichartz-87}. From this point of view, the present note is
certainly only a first step but we should note that our interest for
a Taylor expansion for functionals stems from its potential use in
a cubature scheme in the sense of Kusuoka-Lyons-Victoir \cite{MR2052260}
(which we hope to address in forthcoming work) where it replaces (\ref{eq:stoch taylor})
and, similar to (\ref{eq:stoch taylor}), it might be of independent
interest. Finally, let us mention the classic results of Stroock,
\cite{stroock}, who considered a Taylor expansion of functionals
in the spirit of Malliavin's calculus of variations. However, our
results are different and tailored to the study of the smaller class
of continuous (in uniform norm) and nonanticipative functionals. In
fact, thanks to the extended functional Ito formula of Cont\&Fournie,
it is easy to relax the continuity assumption to include e.g.~functionals
of the quadratic variation (the same arguments then apply since the
domain of the functional changes but only perturbations of the path
are relevant, cf. \cite[theorem 4.1]{ContFournie}). Nevertheless,
we focus our presentation on smooth and continuous functionals since
they already cover interesting applications like pathdependent financial
derivatives, cf.~\cite{Duprire_FIto}, and keep the (necessary) new
notation to a minimum, allowing for a concise presentation of the
present approach.

\section{\label{sec:control}Approximating Functionals of bounded variation
paths}

In this section we adapt some techniques from control theory, introduce
notation and recall a classical existence theorem which demonstrates
that continuous functionals on continuous, bounded variation paths
can be uniformly approximated by linear combinations of iterated integrals,
the (truncated) Chen-Fliess series. In section \ref{sec:Approximating-Functionals-of}
we then extend this result to diffusion functionals using a new approach
with Dupire's functional derivatives. Set 
\[
\Lambda_{1-var}=\left[0,T\right]\times C^{1-var}\left(\left[0,T\right],\mathbb{R}^{d}\right)
\]
with $C^{1-var}\left(\left[0,T\right],\mathbb{R}^{d}\right)$ being
the space of continuous paths $b:\left[0,T\right]\mapsto\mathbb{R}^{d}$
which are of bounded variation, 
\[
\left|b\right|_{1-var;\left[0,T\right]}=\sup_{\mathcal{D}}\sum_{\mathcal{D}}\left|b\left(t_{i}\right)-b\left(t_{i-1}\right)\right|<\infty
\]
with the $\sup_{\mathcal{D}}$ taken over all finite dissections of
$\left[0,T\right]$. We are interested in functionals
\[
F:\Lambda_{1-var}\rightarrow\mathbb{R}
\]
which are nonanticipative in the sense that $F\left(t,b\right)$ depends
only on $\left\{ b\left(r\right),r\leq t\right\} $. A basic example
is the solution map 
\[
F\left(t,b^{1},\ldots,b^{d}\right)=y\left(t\right)
\]
mapping the ``input'' $b=\left(b^{1},\ldots,b^{d}\right)$ to the
real valued solution $y$ of the integral equation 
\begin{eqnarray*}
y\left(t\right) & = & y\left(0\right)+\int_{0}^{t}V\left(y_{r}\right)db_{r}.\\
 & = & y\left(0\right)+\sum_{i=1}^{d}\int_{0}^{t}V_{i}\left(y_{r}\right)db_{r}^{i}
\end{eqnarray*}
Such a map is continuous (provided the vector fields are bounded,
Lipschitz) if $\left[0,T\right]\times C^{1-var}\left(\left[0,T\right],\mathbb{R}^{d}\right)$
is equipped with the product topology. However, for later purposes
it will be convenient to work with the weaker topology on $\Lambda_{1-var}$
induced by the metric
\[
\rho_{1-var}\left(\left(t,b\right),\left(s,e\right)\right)=\left|t-s\right|+\left|a_{t}b-a_{s}e\right|_{1-var;\left[0,T\right]}
\]
where $a_{t}$ denotes the stopping operator 
\begin{eqnarray*}
a_{t}:C^{1-var}\left(\left[0,T\right],\mathbb{R}^{d}\right) & \rightarrow & C^{1-var}\left(\left[0,T\right],\mathbb{R}^{d}\right)\\
\left(a_{t}x\right)\left(r\right) & = & \begin{cases}
x\left(r\right), & r\leq t\\
x\left(t\right), & r>t.
\end{cases}
\end{eqnarray*}
The proposition below then follows from standard arguments.
\begin{prop}
The space $\left(\Lambda_{1-var},\rho_{1-var}\right)$ is a complete
metric space and the topology on $\Lambda_{1-var}$ given by the metric
$\rho_{1-var}$ is strictly weaker than the topology induced by the
product topology (of $\left[0,T\right]\times C^{1-var}\left(\left[0,T\right],\mathbb{R}^{d}\right)$,
the latter equipped with $\left|.\right|_{1-var}$ convergence). 
\end{prop}
We will in the following (unless stated otherwise) always work in
the topology on $\Lambda_{1-var}$ induced by $\rho_{1-var}$. Set
\begin{eqnarray*}
C\left(\Lambda_{1-var}\right) & = & \left\{ F:\Lambda\rightarrow\mathbb{R}\lvert F\text{ is nonanticipative and continuous wrt }\rho_{1-var}\right\} .
\end{eqnarray*}

\begin{defn}
If $b\in C^{1-var}\left(\left[0,T\right],\mathbb{R}^{d}\right)$,
define for every word $\left(i_{1},\ldots,i_{k}\right)\in\left\{ 0,1,\ldots,d\right\} ^{k}$,
and $s,t\in\left[0,T\right]$,$s\leq t$, the iterated integral $\int_{\Delta^{k}\left(s,t\right)}db^{\left(i_{1},\ldots,i_{k}\right)}$
recursively as 
\begin{eqnarray*}
\int_{\Delta^{1}\left(s,t\right)}db^{\left(0\right)} & = & t-s\text{ and }\int_{\Delta^{1}\left(s,t\right)}db^{\left(i\right)}=b^{i}\left(t\right)-b^{i}\left(s\right)\text{ if }i\neq0\\
\int_{\Delta^{k+1}\left(s,t\right)}db^{\left(i_{1},\ldots,i_{k}\right)} & = & \int_{s}^{t}\int_{\Delta^{k}\left(s,r\right)}db^{\left(i_{1},\ldots,i_{k-1}\right)}db_{r}^{i_{k}}.
\end{eqnarray*}

\end{defn}
We now show that the special class of polynomial functionals are dense
in $C\left(\Lambda_{1-var}\right)$.
\begin{defn}
Call a functional $P:\Lambda_{1-var}\rightarrow\mathbb{R}$ a polynomial
functional if there exists a $N\in\mathbb{N}$ and $\left(p_{I}\right)_{I=\left(i_{1},\ldots,i_{k}\right)}\subset\mathbb{R}$
s.t. 
\[
P\left(t,b^{1},\ldots,b^{d}\right)=\sum_{k=1}^{N}\sum_{I\in\left\{ 0,\ldots,d\right\} ^{k}}p_{I}\int_{\Delta^{k}\left(0,t\right)}db^{I}.
\]
\end{defn}
\begin{rem}
\label{rm: algebra}The term polynomial is made more rigorous by using
a morphism with the non commutative algebra $\mathbb{R}\left[X\right]$
over a finite alphabet $X=\{x_{0},\ldots,x_{d}\}$, cf. \cite{fliess1981fonctionnelles}.\end{rem}
\begin{prop}
Every polynomial $P:\Lambda_{1-var}\rightarrow\mathbb{R}$ is continuous
with respect to $\rho_{1-var}$, i.e. an element of $C\left(\Lambda_{1-var}\right)$.
Further, the set of polynomials is a subalgebra of $C\left(\Lambda_{1-var}\right)$.
\end{prop}
The theorem below is a straightforward extension of a well-known theorem
of control theory, cf. \cite[theorem II.5]{sussmann1975semigroup,fliess1981fonctionnelles}
to bounded variation (but not necessarily absolutely continuous) paths
in the topology $\rho_{1-var}$. Note that it relies on the Stone-Weierstrass
theorem. In section \ref{sec:Approximating-Functionals-of} we show
that under stronger assumptions on the functional (and in a semimartingale
context) one can choose the coefficients as functional derivatives
in the sense of Dupire (this is similar to the relation of the classic
Taylor expansion, applicable to smooth functions, to the Weierstrass
approximation theorem which applies to continuous functions).
\begin{thm}
\label{thm:stone-weierst}Let $\mathcal{K}$ be a compact subset of
$C^{1-var}\left(\left[0,T\right],\mathbb{R}^{d}\right)$ and 
\[
F:\left[0,T\right]\times\mathcal{K}\rightarrow\mathbb{R}
\]
be continuous with respect to $\rho_{1-var}$. Then $F$ can be arbitrarily
close approximated in uniform topology by polynomial functionals,
that is $\forall\epsilon>0$ exists a polynomial functional $P:\Lambda_{1-var}\rightarrow\mathbb{R}$
such that
\[
\sup_{\left(t,b\right)\in\left[0,T\right]\times\mathcal{K}}\left|F\left(t,b\right)-P\left(t,b\right)\right|<\epsilon.
\]
\end{thm}
\begin{proof}
The set $\left[0,T\right]\times\mathcal{K}$ is a compact subset of
$\left[0,T\right]\times C^{1-var}\left(\left[0,T\right],\mathbb{R}^{d}\right)$
if the latter set is equipped with the product topology of Euclidean
distance on $\left[0,T\right]$ and $\left|.\right|_{1-var}$. To
see that it is also a compact subset in the topology given by $\rho_{1-var}$,
take a sequence 
\[
\left(t^{n},x^{n}\right)\subset\left[0,T\right]\times\mathcal{K}
\]
By compactness there exists a $\left(t,x\right)\in\left[0,T\right]\times\mathcal{K}$
and a sub-sequence $\left(n_{k}\right)_{k}$ s.t. $\left|t^{n_{k}}-t\right|\rightarrow0$
and $\left|x^{n_{k}}-x\right|_{1-var;\left[0,T\right]}\rightarrow_{k}0$.
Therefore also
\[
\rho\left(t^{n_{k}},x^{n_{k}}\right)=\left|t^{n_{k}}-t\right|+\left|a_{t^{n_{k}}}x_{r}^{n_{k}}-a_{t}x_{r}\right|_{1-var;\left[0,T\right]}\rightarrow_{k}0,
\]
hence $\left[0,T\right]\times\mathcal{K}$ is under the $\rho_{1-var}$
metric a compact subset of $\Lambda_{1-var}$. We now use the Stone-Weierstrass
theorem: $\left(\left[0,T\right]\times\mathcal{K},\rho_{1-var}\right)$
is a compact, Hausdorff space and from proposition we know that the
space of polynomials is a subalgebra of $C\left(\Lambda_{1-var}\right)$,
hence (strictly speaking the space polynomials with domains restricted
to $\left[0,T\right]\times\mathcal{K}$ are) also a subalgebra of
$C\left(\left[0,T\right]\times\mathcal{K}\right)$. Further, the space
of polynomials contains the nonzero, constant functions and it remains
to verify that it separates points. Given two elements of $\left[0,T\right]\times\mathcal{K}$,
\[
\left(t,b^{1},\ldots,b^{d}\right)\neq\left(s,e^{1},\ldots,e^{d}\right),
\]
we can choose a polynomial $P$ s.t. $P\left(t,b^{1},\ldots,b^{d}\right)\neq P\left(s,e^{1},\ldots,e^{d}\right)$:
\begin{itemize}
\item if $s\neq t$: take the functional $P\left(t,b\right)=1$,
\item if $s=t$ and $b^{i}\neq e^{i}$ for some $i\in\left\{ 1,\ldots,d\right\} $:
consider the polynomials
\[
P\left(t,b^{1},\ldots,b^{d}\right)=\int_{\Delta^{k}\left(0,t\right)}db^{I}.
\]
with $I=\bigl(\underbrace{0,\ldots,0}_{k},i\bigr)$ for $k\geq0$.
Then $P\left(t,b-e\right)=\int_{0}^{t}\frac{r^{k}}{k!}d\left(b-e\right)_{r}^{i}.$
Recall that a continuous path $x:\left[0,t\right]\rightarrow\mathbb{R}$
is of finite variation if and only if $x\left(.\right)=\int_{0}^{.}\mu\left(dr\right)$
for a signed measure $\mu$ on $\left[0,t\right]$ with finite mass
and no atoms. The signed measure $\mu$ defines a linear functional
on the space $C\left(\left[0,t\right],\mathbb{R}\right)$. Hence,
if $\int_{0}^{t}p\left(r\right)\mu\left(dr\right)=0$ for all polynomials
$p:\left[0,t\right]\mapsto\mathbb{R}$ this would imply $\int_{\left[0,t\right]}f\left(r\right)\mu\left(dr\right)=0$
$\forall f\in C\left(\left[0,t\right],\mathbb{R}\right)$ by the Stone
Weierstrass theorem on $C\left(\left[0,T\right],\mathbb{R}\right)$,
therefore also $x\left(s\right)=\int_{\left[0,s\right]}\mu\left(dr\right)=0$,
$\forall s\in\left[0,t\right]$. Applying this reasoning to $x=b^{i}-e^{i}$
shows that $P\left(t,b-e\right)=\int_{0}^{t}\frac{r^{k}}{k!}d\left(b-e\right)_{r}^{i}=0$
for every $k\geq0$ leads to a contradiction to $b^{i}\neq e^{i}$.
\end{itemize}
\end{proof}
\begin{rem}
An important generalization of theorem \ref{thm:stone-weierst} to
rough path functionals has been obtained by Lyons%
\footnote{T. Lyons personal communication (cf. thesis of T. Fawcett \cite{Fawcettthesis}
and the transfer report of A. Janssen).%
} et al. \cite{Fawcettthesis} 
\end{rem}

\section{\label{sec:Approximating-Functionals-of}Approximating Functionals
of a diffusion}

Theorem \ref{thm:stone-weierst} applies to bounded variation paths
and gives no description of the approximating functionals. We now
follow a different approach using Dupire's functional derivatives
to show that under some additional smoothness assumption on $F$ one
can also approximate functionals of semi-martingales trajectories
and give an explicit description of the coefficients in the expansion.
Let us fix notation: We are interested in functionals of the form
\[
F:\left[0,T\right]\times D\left(\left[0,T\right],\mathbb{R}^{e}\right)\rightarrow\mathbb{R},
\]
$D\left(\left[0,T\right],\mathbb{R}^{e}\right)$ denoting the space
of càdlàg paths on $\left[0,T\right]$ with values in $\mathbb{R}^{e}$
and we write $\Lambda=\left[0,T\right]\times D\left(\left[0,T\right],\mathbb{R}^{e}\right)$.
Define a metric $\rho_{\infty}$ on $\Lambda$,
\[
\rho_{\infty}\left(\left(t,x\right),\left(s,y\right)\right)=\left|t-s\right|+\left|a_{t}x-a_{s}y\right|_{\infty;\left[0,T\right]}
\]
($a_{t}$ denotes as in section \ref{sec:control} the stopping operator
now applied to cadlag paths) and set 
\[
C\left(\Lambda\right)=\left\{ F:\Lambda\rightarrow\mathbb{R},\text{ progressive measureable wrt to \ensuremath{\mathcal{F}}, continuous wrt to }\rho_{\infty}\right\} ,
\]
$\mathcal{F}$ being the $\sigma$-algebra generated by the coordinate
process on $D\left(\left[0,T\right],\mathbb{R}^{e}\right)$. Following
Dupire \cite{Duprire_FIto}, we say that $F\in C\left(\Lambda\right)$
has a time derivative $\partial_{0}F\left(t,x\right)$ at $\left(t,x\right)\in\Lambda$
if the limit 
\[
\partial_{0}F\left(t,x\right)=\lim_{\epsilon\downarrow0}\frac{1}{\epsilon}\left(F\left(t+\epsilon,a_{t}x\right)-F\left(t,x\right)\right)
\]
exists and a space derivative $\partial_{i}F\left(t,x\right)$ in
direction $e_{i},i=1,\ldots,e,$ ($e_{i}$ denoting the standard basis
on $\mathbb{R}^{e}$) at $\left(t,x\right)\in\Lambda$ if the limit
\[
\partial_{i}F\left(t,x\right)=\lim_{\epsilon\rightarrow0}\frac{1}{\epsilon}\left(F\left(t,x+\epsilon e_{i}1_{\cdot\geq t}\right)-F\left(t,x\right)\right)
\]
exists. Similarly, define the higher order derivatives $\partial_{i}\partial_{j}F$
etc. Denote $C^{m,n}\left(\Lambda\right)$ the subset of $C\left(\Lambda\right)$
of functionals which are at least $m$ times differentiable in time,
at least $n$ times differentiable in space and continuous with respect
to $\rho_{\infty}$. More formally, we define a set of multi-indices
by letting
\[
\mathcal{E}_{m,n}:=\left\{ (\beta_{1},\ldots,\beta_{i})\in\bigcup_{k=0}^{\infty}\{0,\ldots,e\}^{k}:card(j:\beta_{j}=0)\leq m\text{ and }card(j:\beta_{j}\neq0)\leq n\right\} 
\]
(by convention $\left\{ 0,\ldots,e\right\} ^{0}=\emptyset)$ and say
$F\in C^{m,n}\left(\Lambda\right)$ if $\partial_{\alpha_{1}}\dots\partial_{\alpha_{i}}F\in C\left(\Lambda\right)$
for all $(\alpha_{1},\ldots,\alpha_{i})\in\mathcal{E}_{m,n}$. Furthermore
we define sets of bounded functions with bounded derivatives by 
\[
C_{b}^{m,n}\left(\Lambda\right)=\left\{ F\in C^{m,n}\left(\Lambda\right):\sup_{\left(r,y\right)\in\Lambda}\left|\partial_{\alpha_{1}}\dots\partial_{\alpha_{i}}F\left(r,y\right)\right|<\infty,(\alpha_{1},\ldots,\alpha_{n})\in\mathcal{E}_{m,n}\right\} .
\]
The functional derivatives behave similar to the standard derivatives,
we make repeatedly use of two properties.
\begin{lem}
\label{lem:product}Let $m,n\in\mathbb{N}$ and suppose $F,G\in C_{b}^{m,n}\left(\Lambda\right).$ 
\begin{enumerate}
\item Then $FG\in C_{b}^{m,n}\left(\Lambda\right)$ and $\partial_{i}\left(FG\right)=\left(\partial_{i}F\right)G+F\left(\partial_{i}G\right)$
for $i\in\left\{ 0,1,\ldots,e\right\} $.
\item If $F\left(t,x\right)=f\left(t,x_{t}\right)$ for $f\in C^{1,1}\left(\left[0,T\right]\times\mathbb{R}^{e},\mathbb{R}\right)$,
then $F\in C^{1,1}\left(\Lambda\right)$ and $\left(\partial_{0}F\right)\left(t,x\right)=\frac{df}{dt}\left(t,x_{t}\right)$
and $\left(\partial_{i}F\right)\left(t,x\right)=\frac{df}{dx^{i}}\left(t,x_{t}\right)$
for $i\in\left\{ 1,\ldots,e\right\} $ 
\end{enumerate}
\end{lem}
\begin{thm}[Dupire's functional It\={o} formula]
\label{thm:Dupire Ito}Let $F\in C_{b}^{1,2}\left(\Lambda\right)$,
$Y$ a continuous, $\mathbb{R}^{e}$-valued semi-martingale. Then
$F\left(Y\right)$ is a continuous semi-martingale and a.s. 
\[
F_{t}\left(Y\right)-F_{s}\left(Y\right)=\int_{s}^{t}\partial_{0}F_{r}\left(Y\right)dr+\sum_{i=1}^{e}\int_{s}^{t}\partial_{i}F_{r}\left(Y\right)dY_{r}^{i}+\frac{1}{2}\sum_{i,j=1}^{e}\int_{s}^{t}\partial_{ij}F_{r}\left(Y\right)d\left[Y\right]_{r}^{ij}.
\]
\end{thm}
\begin{proof}
See \cite{Duprire_FIto} and \cite{ContFournie}.\end{proof}
\begin{rem}
Above formula also holds when the boundedness is replaced by boundedness
on bounded sets and $F$ depends only continuous wrt to $Y$ and an
extra variable e.g. the quadratic variation of $Y$, cf. \cite{ContFournie}.
\end{rem}
We actually need the Stratonovich version with the Stratonovich integral
as usual defined as 
\[
\int X\circ dY:=\int XdY+\frac{1}{2}\left[X,Y\right].
\]
Above formula then reads
\begin{cor}
\label{cor:functiional strat}Let $F\in C_{b}^{1,3}\left(\Lambda\right)$,
$Y$ a continuous, $\mathbb{R}^{e}$-valued semi-martingale. Then
we have a.s.
\[
F_{t}\left(Y\right)-F_{s}\left(Y\right)=\int_{s}^{t}\partial_{0}F_{r}\left(Y\right)dr+\sum_{i=1}^{e}\int_{s}^{t}\partial_{i}F_{r}\left(Y\right)\circ dY_{r}^{i}.
\]
\end{cor}
\begin{proof}
We need to show that for $i\in\left\{ 1,\ldots,e\right\} $, 
\[
\left[\partial_{i}F,Y^{i}\right]_{t}=\sum_{j=1}^{e}\int_{s}^{t}\partial_{j}\partial_{i}F_{r}\left(Y\right)d\left[Y\right]_{r}^{ji}.
\]
By assumption $\partial_{i}F\in C_{b}^{1,2}$, hence from theorem
\ref{thm:Dupire Ito}, 
\[
d\left(\partial_{i}F_{t}\right)=\partial_{0}\left(\partial_{i}F_{t}\right)dt+\sum_{j=1}^{e}\partial_{j}\partial_{i}F_{t}dY_{t}^{j}+\frac{1}{2}\sum_{j,k=1}^{e}\left(\partial_{j}\partial_{k}\partial_{i}F_{t}\right)d\left[Y\right]_{t}^{jk},
\]
and it follows that, 
\[
\left[\partial_{i}F,Y^{i}\right]_{t}=\sum_{j=1}^{e}\left[\int\partial_{j}\partial_{i}FdY^{j},Y^{i}\right]_{t}
\]
which equals $\sum_{j=1}^{e}\int\partial_{j}\partial_{i}F\left(t,Y\right)d\left[Y\right]_{t}^{ji}$.
\end{proof}

\subsection{Vector fields as derivations of functionals}

Let $M$ be a smooth manifold, then one can regard a vector field
$V$ on $M$ as a derivation, $V\in Der_{\mathbb{R}}\left(C^{\infty}\left(M\right)\right)$.
We need to find a similar concept for the (infinite dimensional) space
$\Lambda$ which respects the additional structure (of being the product
of a time coordinate and a pathspace). It seems natural to define
for a given vector field $V\in C_{b}^{\infty}\left(\mathbb{R}^{e},\mathbb{R}^{e}\right)$,
\[
V\cdot F\left(t,y\right)=\sum_{i=1}^{e}V^{i}\left(y_{t}\right)\partial_{i}F\left(t,y\right)\text{ for }F\in C_{b}^{\infty,\infty}\left(\Lambda\right).
\]
However, above formulation does not take into account the time decay
in functionals measured by the functional time derivative $\partial_{0}$.
We therefore instead consider maps $\overline{V}:\mathbb{R}^{e}\rightarrow\mathbb{R}^{e+1}$
with an additional $0$th coordinate $\overline{V}=\left(\overline{V}^{j}\right)_{j=0}^{e}$
and let $\overline{V}$ act on $C_{b}^{\infty,\infty}\left(\Lambda\right)$
as 
\begin{equation}
\overline{V}\cdot F\left(t,y\right)=\sum_{i=0}^{e}\overline{V}^{i}\left(y_{t}\right)\partial_{i}F\left(t,y\right)\text{ for }F\in C_{b}^{\infty,\infty}\left(\Lambda\right).\label{eq:vf as derivation}
\end{equation}

\begin{lem}[\textbf{and Definition}]
\label{lem:Given-a-map}Given a map $\overline{V}\in C_{b}^{\infty}\left(\mathbb{R}^{e},\mathbb{R}^{e+1}\right)$,
we can identify $\overline{V}$ as an element of $Der_{\mathbb{R}}\left(C_{b}^{\infty,\infty}\left(\Lambda\right)\right),$
the derivations on the $\mathbb{R}-$algebra $C_{b}^{\infty,\infty}\left(\Lambda\right)$,
by setting 
\[
\overline{V}\cdot F\left(t,y\right):=\sum_{i=0}^{e}\overline{V}^{i}\left(y_{t}\right)\partial_{i}F\left(t,y\right)\text{ for }F\in C_{b}^{\infty,\infty}\left(\Lambda\right).
\]
 \end{lem}
\begin{proof}
By lemma \ref{lem:product}, $\overline{V}$ maps $C_{b}^{\infty,\infty}\left(\Lambda\right)$
to $C_{b}^{\infty,\infty}\left(\Lambda\right)$ functionals. One need
to verify that $C_{b}^{\infty,\infty}\left(\Lambda\right)$ is an
$\mathbb{R}$-algebra and that $\overline{V}\cdot\left(FG\right)=\left(\overline{V}\cdot F\right)G+F\left(\overline{V}\cdot G\right)$,
but this follows again directly from lemma \ref{lem:product}.\end{proof}
\begin{prop}
\label{cor:Ito for FDE}Let $X=\left(X^{i}\right)_{i=0}^{d}$ be a
continuous, semi-martingale with $X_{t}^{0}=t$,$V_{i}=\left(V_{i}^{j}\right)_{j=1}^{e}\in C_{b}^{\infty}\left(\mathbb{R}^{e},\mathbb{R}^{e}\right)$,
$i\in\left\{ 0,1,\ldots,d\right\} $ and $Y=\left(Y^{j}\right)_{j=1}^{d}$
be the unique strong solution of the SDE 
\begin{eqnarray*}
dY_{t} & = & V\left(Y_{t}\right)\circ dX\\
 & = & V_{0}\left(Y_{t}\right)dt+\sum_{i=1}^{d}V_{i}\left(Y_{t}\right)\circ dX_{t}^{i}.
\end{eqnarray*}
Define $\overline{V}_{i}\in C_{b}^{\infty}\left(\mathbb{R}^{e},\mathbb{R}^{e+1}\right)$
as%
\footnote{$\delta_{i,j}$ the usual Kronecker delta, $\delta_{ij}=1$ if $i=j,$
otherwise equal $0$%
} 
\[
\overline{V}_{i}=\left(\delta_{0i},V_{i}^{1},\ldots,V_{i}^{e}\right)^{T}\text{ for }i\in\left\{ 0,\ldots,d\right\} .
\]
Then for every \textup{$F\in C_{b}^{\infty,\infty}\left(\Lambda\right)$
}we have a.s. \textup{
\begin{eqnarray*}
F_{t}\left(Y\right) & = & F_{s}\left(Y\right)+\sum_{i=0}^{d}\int_{s}^{t}\overline{V}_{i}\cdot F_{r}\left(Y\right)\circ dX_{r}^{i}\\
 & = & F_{s}\left(Y\right)+\int_{s}^{t}\overline{V}\cdot F_{r}\left(Y\right)\circ dX_{r}.
\end{eqnarray*}
}\end{prop}
\begin{proof}
Applying corollary \ref{cor:functiional strat} to $F$ gives 
\begin{eqnarray*}
F_{t}\left(Y\right) & = & F_{s}\left(Y\right)+\int_{s}^{t}\partial_{0}F_{r}\left(Y\right)dr+\sum_{j=1}^{e}\int_{s}^{t}\partial_{j}F_{r}\left(Y\right)\circ dY_{r}^{j}\\
 & = & F_{s}\left(Y\right)+\int_{s}^{t}\partial_{0}F_{r}\left(Y\right)dr+\sum_{j=1}^{e}\sum_{i=0}^{d}\int_{s}^{t}\partial_{j}F_{r}\left(Y\right)V_{i}^{j}\left(Y_{r}\right)\circ dX_{r}^{i}\\
 & = & F_{s}\left(Y\right)+\int_{s}^{t}\partial_{0}F_{r}\left(Y\right)dr+\sum_{j=1}^{e}\int_{s}^{t}\partial_{j}F_{r}\left(Y\right)V_{0}^{j}\left(Y_{r}\right)\circ dr\\
 &  & +\sum_{j=1}^{e}\sum_{i=1}^{d}\int_{s}^{t}\partial_{j}F_{r}\left(Y\right)V_{i}^{j}\left(Y_{r}\right)\circ dX_{r}^{i}
\end{eqnarray*}
Identifying $\overline{V}_{0}$ and $\overline{V}_{1},\ldots,\overline{V}_{d}$
as a derivation, as in lemma \ref{lem:Given-a-map} (recall that $\overline{V}_{i}^{0}=0$
for $i\neq0$), this becomes
\begin{eqnarray*}
F_{t}\left(Y\right) & = & F_{s}\left(Y\right)+\int_{s}^{t}\overline{V}_{0}\cdot F_{r}\left(Y\right)\circ dX_{r}^{0}+\sum_{i=1}^{d}\int_{s}^{t}\overline{V}_{i}\cdot F_{r}\left(Y\right)\circ dX_{r}^{i}
\end{eqnarray*}

\end{proof}
To identify compositions of the vector fields we define the set of
multi-indices $\mathcal{A}$ by
\[
\mathcal{A}:=\bigcup_{k=1}^{\infty}\{0,\ldots,d\}^{k}
\]
and let $I=(\alpha_{1},\ldots,\alpha_{k})\in\mathcal{A}$ be a multi-index.
Given $I\in\mathcal{A}$ set $\left|I\right|=card\left(I\right)$
and also define $\Vert I\Vert:=|I|+card(j:\alpha_{j}=0)$. Finally
let
\[
\mathcal{A}(j)=\{I\in\mathcal{A}:\Vert I\Vert\leq j\}.
\]
Given a multi-index $I=(\alpha_{1},\ldots,\alpha_{k})\in\mathcal{A}$
we define $\overline{V}_{I}:=\overline{V}_{\alpha_{1}}\cdots\overline{V}_{\alpha_{k}}$.
\begin{defn}
Set 
\[
\Delta_{k}\left(s,t\right)=\left\{ \left(t_{1},\ldots,t_{k}\right)\in\left[s,t\right]^{k},s\leq t_{1}\leq\cdots\leq t_{k}\leq t\right\} 
\]
Let $X=\left(X^{i}\right)_{i=0}^{d}$ be a continuous semi-martingale
with $X_{t}^{0}=t$. Define the iterated Stratonovich integrals of
$X$ as 
\[
\int_{\Delta_{\left|I\right|}\left(s,t\right)}\circ dX^{I}=\int_{\Delta_{k}\left(s,t\right)}\circ dX_{t_{1}}^{\alpha_{1}}\cdots\circ dX_{t_{k}}^{\alpha_{k}}\text{ with }I=\left(\alpha_{1},\ldots,\alpha_{k}\right)\in\mathcal{A}.
\]

\end{defn}

\subsection{The functional Taylor series and a remainder estimate}

We can now formulate our main theorem. The reader will notice that
applied to a functional of the form $F\left(t,x\right)=f\left(t,x_{t}\right)$,
one recovers the usual stochastic Taylor expansion, \cite{kloeden-platen-92,MR2052260}.
\begin{thm}
\label{pro:Euler approx}Let $X=\left(X^{i}\right)_{i=0}^{d}$ be
a continuous, semi-martingale with $X_{t}^{0}=t$, $V=\left(V_{i}\right)_{i=1}^{d}$
with $V_{i}\in C_{b}^{\infty}\left(\mathbb{R}^{e},\mathbb{R}^{e}\right)$
and let $Y$ be the unique strong solution of the SDE 
\begin{eqnarray*}
dY_{t} & = & V\left(Y_{t}\right)\circ dX_{t}\equiv\sum_{i=0}^{d}V_{i}\left(Y_{t}\right)\circ dX_{t}^{i}.
\end{eqnarray*}
Define $\overline{V}_{i}\in C_{b}^{\infty}\left(\mathbb{R}^{e},\mathbb{R}^{e+1}\right)$
as 
\[
\overline{V}_{i}=\left(\delta_{0i},V_{i}^{1},\ldots,V_{i}^{e}\right)^{T}\text{ for }i\in\left\{ 0,\ldots,d\right\} .
\]
Fix $m\in\mathbb{N}$. Then for every $F\in C_{b}^{\infty,\infty}\left(\Lambda\right)$
we have 
\begin{eqnarray*}
F_{t}\left(Y\right)-F_{s}\left(Y\right) & = & \sum_{I\in\mathcal{A}(m)}\overline{V}_{I}\cdot F_{s}\left(Y\right)\int_{\Delta_{|I|}\left(s,t\right)}\circ dX^{I}+R_{st}^{m}.
\end{eqnarray*}
where 
\[
R_{st}^{m}=\sum_{\begin{array}{c}
(i_{1},\ldots,i_{N})\notin\mathcal{A}(m)\\
(i_{2},\ldots,i_{N})\in\mathcal{A}(m)
\end{array}}\int_{s<r_{1}<\ldots<r_{N}<t}\overline{V}_{i_{1}}\cdots\overline{V}_{i_{N}}\cdot F_{r}\left(Y\right)\circ dX_{r_{1}}^{i_{1}}\cdots\circ dX_{r_{N}}^{i_{N}}.
\]
\end{thm}
\begin{proof}
Applying proposition \ref{cor:Ito for FDE} gives 
\begin{eqnarray*}
F_{t}\left(Y\right) & = & F_{s}\left(Y\right)+\sum_{i=0}^{d}\int_{s}^{t}\overline{V}_{i}\cdot F_{r}\left(Y\right)\circ dX_{r}^{i}.
\end{eqnarray*}
Now by assumption on $F$ and $V$, the functional $\overline{V}_{i}\cdot F$
is again in $C_{b}^{\infty,\infty}$ and applying proposition \ref{cor:Ito for FDE}
to $\overline{V}_{i}\cdot F$ gives 
\begin{eqnarray*}
F_{t}\left(Y\right) & = & F_{s}\left(Y\right)+\sum_{i_{1}=0}^{d}\overline{V}_{i_{1}}\cdot F_{s}\left(Y\right)\int_{s}^{t}\circ dX^{i_{1}}\\
 &  & +\sum_{i_{1},i_{2}=0}^{d}\int_{s<r_{2}<r_{1}<t}\overline{V}_{i_{2}}\cdot\overline{V}_{i_{1}}\cdot F_{r_{2}}\left(Y\right)\circ dX_{r_{2}}^{i_{2}}\circ dX_{r_{1}}^{i_{1}}.
\end{eqnarray*}
Iterating this procedure, we finally arrive at 
\begin{eqnarray*}
F_{t}\left(Y\right) & = & F_{s}\left(Y\right)+\sum_{I\in\mathcal{A}\left(m\right)}\overline{V}_{I}\cdot F_{s}\left(Y\right)\int_{\Delta_{|I|}\left(s,t\right)}\circ dX_{st}^{I}\\
 &  & +\sum_{\begin{array}{c}
(i_{1},\ldots,i_{N})\notin\mathcal{A}(m)\\
(i_{2},\ldots,i_{N})\in\mathcal{A}(m)
\end{array}}\int_{s<r_{1}<\ldots<r_{N}<t}\overline{V}_{i_{1}}\cdots\overline{V}_{i_{N}}\cdot F_{r_{1}}\left(Y\right)\circ dX_{r_{1}}^{i_{1}}\cdots\circ dX_{r_{N}}^{i_{N}}.
\end{eqnarray*}
\end{proof}
\begin{rem}
The choice of truncation in the Taylor approximation in Theorem \ref{pro:Euler approx}
reflects how the semi-martingale part of the noise $X$ scales compared
to the time component $X_{t}^{0}=t$. If $X$ is a Brownian motion,
then (cf.~\cite{MR2052260}) for $I=\left(\alpha_{1},\ldots,\alpha_{k}\right)$
\[
\int_{\Delta_{\left(0,t\right)}^{k}}\circ dX^{I}\text{ is equal in law to }\sqrt{t}^{k+card\left\{ j,\alpha_{j}=0\right\} }\int_{\Delta_{\left(0,1\right)}^{k}}\circ dX^{I}.
\]
This explains the particular truncation of the Taylor expansion via
the multi-index sets $\mathcal{A}\left(m\right)$ and allows to arrive
at the $L^{2}$ error estimate in corollary \ref{cor:error} which
is important in numerical applications, e.g. \cite{MR2052260}.
\begin{rem}
Even if there is no drift vector field in the SDE, that is in above
notation $V_{0}=\left(0,\ldots,0\right)^{T}$, we still have to deal
with iterated integrals which involve $dX_{t}^{0}=dt$. This is in
contrast to the stochastic Taylor expansion for functions $f:\mathbb{R}^{e}\rightarrow\mathbb{R}^{e}$
due to the time decay in the functional $F\in C_{b}^{\infty,\infty}\left(\Lambda\right)$
which is picked up by the functional derivative $\partial_{0}$ via
the first coordinate of $\overline{V}_{0}=\left(1,0,\ldots,0\right)^{T}$.
\end{rem}
\end{rem}
Of special importance is the case of Brownian noise in which case
we can give a $L^{2}$ estimate for the remainder term.
\begin{cor}
\label{cor:error}Let $\left(B^{i}\right)_{i=1}^{d}$ be a $d$-dimensional
Brownian motion and set $B=\left(B^{i}\right)_{i=0}^{d}$ with $B_{t}^{0}=t$.
Let $V=\left(V_{i}\right)_{i=1}^{d}$ with $V_{i}\in C_{b}^{\infty}\left(\mathbb{R}^{e},\mathbb{R}^{e}\right)$
and $Y$ be the unique strong solution of the Stratonovich SDE with
drift, 
\begin{eqnarray*}
dY_{t} & = & V\left(Y_{t}\right)\circ dB_{t}\equiv\sum_{i=0}^{d}V_{i}\left(Y_{t}\right)\circ dB_{t}^{i}.
\end{eqnarray*}
Define $\overline{V}_{i}\in C_{b}^{\infty}\left(\mathbb{R}^{e},\mathbb{R}^{e+1}\right)$
as $\overline{V}_{i}=\left(\delta_{0i},V_{i}^{1},\ldots,V_{i}^{e}\right)^{T}\text{ for }i\in\left\{ 0,\ldots,d\right\} .$
Fix $m\in\mathbb{N}$. Then for every $F\in C_{b}^{\infty,\infty}\left(\Lambda\right)$
we have 
\begin{eqnarray*}
F_{t}\left(Y\right)-F_{s}\left(Y\right) & = & \sum_{I\in\mathcal{A}(m)}\overline{V}_{I}\cdot F_{s}\left(Y\right)\int_{\Delta_{|I|}\left(s,t\right)}\circ dB^{I}+R_{st}^{m}.
\end{eqnarray*}
Moreover, there exists a constant $c=c\left(d,m\right)$, only depending
on $d$ and $m$, such that 
\[
\mathbb{E}\left[\left|R_{st}^{m}\right|^{2}\right]^{1/2}\leq c\sum_{j=m+1}^{m+2}\sup_{I\in\mathcal{A}(j)\setminus\mathcal{A}(j-1)}\left|t-s\right|^{\Vert I\Vert/2}\sup_{\left(r,y\right)\in\Lambda}\left|\overline{V}_{I}\cdot F\left(r,y\right)\right|.
\]
\end{cor}
\begin{proof}
Given $G\in C_{b}^{1,2}\left(\Lambda\right)$, $k\in\mathbb{N}$ and
$I=\left(\alpha_{1},\ldots,\alpha_{k}\right)\in\mathcal{A}$. We show
that there exists a constant $C=C\left(d,k\right)$, only depending
on $d$ and $k$, s.t. $\forall t>0$ 
\begin{eqnarray*}
 &  & \mathbb{E}\left(\int_{0<t_{1}<\dots<t_{k}<t}G(t_{1},Y)\circ dB_{t_{1}}^{\alpha_{1}}\dots\circ dB_{t_{k}}^{\alpha_{k}}\right)^{2}\\
 & \leq & C\left(t^{\|(\alpha_{1},\ldots,\alpha_{k})\|}\|G\|_{\infty}^{2}+(1-\delta_{0,\alpha_{1}})t^{\|(\alpha_{1},\ldots,\alpha_{k})\|+1}\|\bar{V}_{\alpha_{1}}G\|_{\infty}^{2}\right).
\end{eqnarray*}
The result then follows from Theorem \ref{pro:Euler approx} by applying
this to $G=\overline{V}_{i_{1}}\cdots\overline{V}_{i_{N}}\cdot F$
(without loss of generality one can take $s=0$). We assume for induction
that for any $j\leq k$, 
\begin{eqnarray*}
 &  & \mathbb{E}\left(\int_{0<t_{1}<\dots<t_{k}<t}G(t_{1},Y)\circ dB_{t_{1}}^{\alpha_{1}}\dots\circ dB_{t_{j}}^{\alpha_{j}}\right)^{2}\\
 & \leq & C\left(j,d\right)\left(t^{\|(\alpha_{1},\ldots,\alpha_{j})\|}\|G\|_{\infty}^{2}+(1-\delta_{0,\alpha_{1}})t^{\|(\alpha_{1},\ldots,\alpha_{j})\|+1}\|\bar{V}_{\alpha_{1}}G\|_{\infty}^{2}\right).
\end{eqnarray*}
for all $t>0$. We will prove the claim for $I=(\alpha_{1},\ldots,\alpha_{k+1})$.
First assume that $\alpha_{k+1}=0$. In this case we have 
\begin{eqnarray*}
 &  & \mathbb{E}\left(\int_{0}^{t}\int_{0<t_{1}<\dots<t_{k}<t_{k+1}}G(t_{1},Y)\circ dB_{t_{1}}^{\alpha_{1}}\dots\circ dB_{t_{k}}^{\alpha_{k}}dt_{k+1}\right)^{2}\\
 & = & \mathbb{E}\left(t\int_{0}^{1}\int_{0<t_{1}<\dots<t_{k}<t_{k+1}t}G(t_{1},Y)\circ dB_{t_{1}}^{\alpha_{1}}\dots\circ dB_{t_{k}}^{\alpha_{k}}dt_{k+1}\right)^{2}\\
 & \leq & t^{2}\mathcal{\mathbb{E}}\left[\int_{0}^{1}\left(\int_{0<t_{1}<\dots<t_{k}<t_{k+1}t}G(t_{1},Y)\circ dB_{t_{1}}^{\alpha_{1}}\dots\circ dB_{t_{k}}^{\alpha_{k}}\right)^{2}dt_{k+1}\right]\\
 & = & t^{2}\int_{0}^{1}\mathbb{E}\left(\int_{0<t_{1}<\dots<t_{k}<t_{k+1}t}G(t_{1},Y)\circ dB_{t_{1}}^{\alpha_{1}}\dots\circ dB_{t_{k}}^{\alpha_{k}}\right)^{2}dt_{k+1}
\end{eqnarray*}
by using the Jensen inequality. Assuming the inductive hypothesis
we estimate that last line as 
\begin{eqnarray*}
 & \leq & C(k,d)t^{2}\int_{0}^{1}\left((ts)^{\|(\alpha_{1},\ldots,\alpha_{k})\|}\|G\|_{\infty}^{2}+(1-\delta_{0,\alpha_{1}})(ts)^{\|(\alpha_{1},\ldots,\alpha_{k})\|+1}\|\bar{V}_{\alpha_{1}}G\|_{\infty}^{2}\right)ds\\
 & \leq & C(k+1,d)\left(t^{\|(\alpha_{1},\ldots,\alpha_{k+1})\|}\|G\|_{\infty}^{2}+(1-\delta_{0,\alpha_{1}})t^{\|(\alpha_{1},\ldots,\alpha_{k+1})\|+1}\|\bar{V}_{\alpha_{1}}G\|_{\infty}^{2}\right).
\end{eqnarray*}
 Now suppose $\alpha_{k+1}\neq0$. We have if $\alpha_{k}\neq0$ (it
easy to see that $\alpha_{k}=0$ follows from a similar calculation)
\begin{eqnarray*}
 &  & \mathbb{E}\left(\int_{0<t_{1}<\dots<t_{k+1}<t}G(t_{1},Y)\circ dB_{t_{1}}^{\alpha_{1}}\dots\circ dB_{k+1}^{\alpha_{k+1}}\right)^{2}\\
 & = & \mathbb{E}\left(\int_{0<t_{1}<\dots<t_{k+1}<t}G(t_{1},Y)\circ dB_{t_{1}}^{\alpha_{1}}\dots\circ dB_{k}^{\alpha_{k}}dB_{k+1}^{\alpha_{k+1}}\right.\\
 &  & +\left.\frac{1}{2}\left[\int_{0<t_{1}<\dots<t_{k}<.}G(t_{1},Y)\circ dB_{t_{1}}^{\alpha_{1}}\dots\circ dB_{t_{k}}^{\alpha_{k}},B^{\alpha_{k+1}}\right]_{t}\right)^{2}\\
 & \leq & 2\mathbb{E}\left(\int_{0<t_{1}<\dots<t_{k+1}<t}G(t_{1},Y)\circ dB_{t_{1}}^{\alpha_{1}}\dots\circ dB_{t_{k}}^{\alpha_{k}}dB_{t_{k+1}}^{\alpha_{k+1}}\right)^{2}\\
 &  & +2\delta_{\alpha_{k},\alpha_{k+1}}\mathbb{E}\left(\int_{0<t_{1}<\dots<t_{k-1}<s<t}G(t_{1},Y)\circ dB_{t_{1}}^{\alpha_{1}}\dots\circ dB_{k-1}^{\alpha_{k-1}}ds\right)^{2}
\end{eqnarray*}
Note that we can bound the second term in the sum using the same arguments
as in the case $\alpha_{k+1}=0$. For the first term we estimate the
Ito integral, 
\begin{eqnarray*}
 &  & \mathbb{E}\left(\int_{0<t_{1}<\dots<t_{k+1}<t}G(t_{1},Y)\circ dB_{t_{1}}^{\alpha_{1}}\dots\circ dB_{t_{k}}^{\alpha_{k}}dB_{t_{k+1}}^{\alpha_{k+1}}\right)^{2}\\
 & = & \int_{0}^{t}\mathbb{E}\left(\int_{0<t_{1}<\dots<t_{k+1}}G(t_{1},Y)\circ dB_{t_{1}}^{\alpha_{1}}\dots\circ dB_{t_{k}}^{\alpha_{k}}\right)^{2}dt_{k+1}\\
 & \leq & C(k,d)\int_{0}^{t}\left(t_{k+1}^{\|(\alpha_{1},\ldots,\alpha_{k})\|}\|G\|_{\infty}^{2}+(1-\delta_{0,\alpha_{1}})t_{k+1}^{\|(\alpha_{1},\ldots,\alpha_{k})\|+1}\|\bar{V}_{\alpha_{1}}G\|_{\infty}^{2}\right)dt_{k+1}\\
 & \leq & C(k+1,d)\left(t^{\|(\alpha_{1},\ldots,\alpha_{k})\|+1}\|G\|_{\infty}^{2}+(1-\delta_{0,\alpha_{1}})t^{\|(\alpha_{1},\ldots,\alpha_{k})\|+2}\|\bar{V}_{\alpha_{1}}G\|_{\infty}^{2}\right)
\end{eqnarray*}
using the inductive hypothesis. Finally we prove the base case of
the induction: Recall that $Y$ is $\mathbb{R}^{e}$-valued and $\overline{V}_{i}=(0,V_{i}^{1},\ldots,V_{i}^{e})$.
Similar to proposition \ref{cor:Ito for FDE}, we see 
\begin{eqnarray*}
G_{t}\left(Y\right)-G_{0}\left(Y\right) & = & \sum_{j=1}^{d}\int_{0}^{t}\overline{V}_{j}\cdot G(t,Y)dB_{t}^{i}+BV_{t}
\end{eqnarray*}
with $BV$ a continuous bounded variation process. Thus for $\alpha_{1}\neq0$
we have 
\[
d\left[G(\cdot,Y),B^{\alpha_{1}}\right]_{t}=\sum_{i=1}^{d}\overline{V}_{i}\cdot G(t,Y)d\left[B^{i},B^{\alpha_{1}}\right]_{t}=\overline{V}_{\alpha_{1}}\cdot G(t,Y)dt.
\]
Hence, 
\begin{eqnarray*}
\int_{0}^{t}G(s,Y)\circ dB_{s}^{\alpha_{1}} & = & \int_{0}^{t}G(r,Y)dB_{r}^{\alpha_{1}}+\frac{1}{2}\left[G(\cdot,Y),B^{\alpha_{1}}\right]_{t}\\
 & = & \int_{0}^{t}G(r,Y)dB_{r}^{\alpha_{1}}+\frac{1}{2}\int_{0}^{t}\overline{V}_{\alpha_{1}}G(r,Y)dr
\end{eqnarray*}
which implies the base case of the induction for $\alpha_{1}\neq0$.
The estimate for the case $\alpha_{1}=0$, i.e. $B_{t}^{\alpha_{1}}=t$
is obvious. 
\end{proof}
To demonstrate that above theorem leads to concrete, easy to use,
estimates we conclude with a toy example which can be verified (by
a longer calculation) with standard Ito calculus.
\begin{example}
Take a $1$-dimensional Brownian motion $B$, $f,g,V_{1}\in C_{b}^{\infty}\left(\mathbb{R},\mathbb{R}\right)$
and let $dY=V_{1}\left(Y_{t}\right)\circ dB$. Consider 
\[
F\left(t,x\right)=f\left(\int_{0}^{t}g\left(x_{r}\right)dr\right).
\]
We immediately get $\partial_{0}F\left(t,x\right)=g\left(x_{t}\right)f'\left(\int_{0}^{t}g\left(x_{r}\right)dr\right)$,
$\partial_{1}F\left(t,x\right)=0$, $\partial_{1}\partial_{0}F\left(t,x\right)=g'\left(x_{t}\right)f'\left(\int_{0}^{t}g\left(x_{r}\right)dr\right)$
and similarly one can calculate higher derivatives. The first nontrivial
case of corollary \ref{cor:error} applied to this example is the
step $m=3$ approximation, 
\begin{eqnarray*}
F_{t}\left(Y\right)-F_{s}\left(Y\right) & = & \partial_{0}F_{s}\left(Y\right)\int_{s<r<t}\circ dr+V_{1}\left(Y_{s}\right)\partial_{1}\partial_{0}F_{s}\left(Y\right)\int_{s<r_{1}<r_{2}<t}\circ dB_{r_{1}}\circ dr_{2}+R_{st}^{3}
\end{eqnarray*}
with $\mathbb{E}\left[\left|R_{st}^{3}\right|^{2}\right]^{1/2}\leq c\sup_{\begin{subarray}{c}
I\in\left\{ \left(1,1,1,0\right),\left(1,1,0\right),\left(0,1,0\right),\left(0,0\right)\right\} \\
\left(r,y\right)\in\Lambda
\end{subarray}}\left(\left|t-s\right|^{\bigl\Vert I\bigr\Vert/2}\left|\overline{V}_{I}\cdot F\left(r,y\right)\right|\right)$.
\end{example}

\subsection*{Acknowledgment. }

Christian Litterer and Harald Oberhauser were partially supported
by the European Unions Seventh Framework Programme, ERC grant agreement
258237.

\bibliographystyle{plain}
\bibliography{/home/hd/Dropbox/projects/BibteX/roughpaths,/homes/stoch/oberhaus/Dropbox/projects/BibteX/roughpaths}

\end{document}